\documentclass[11 pt]{amsart}
\usepackage{amscd,amsfonts,amssymb,amsmath}
\usepackage{hyperref}

\usepackage{tikz}
\usepackage[margin=3.7 cm]{geometry}

\newtheorem{theorem}{Theorem}[section]

\newtheorem{lemma}[theorem]{Lemma}
\newtheorem{prop}[theorem]{Proposition}

\theoremstyle{definition}

\newtheorem{example}[theorem]{Example}
\newtheorem{remark}[theorem]{Remark}
\numberwithin{equation}{subsection}
\theoremstyle{plain}

\newtheorem{conjecture}{Conjecture}

\newtheorem{problem}{Problem}
\newtheorem{corollary}[theorem]{Corollary}

\def\Z{\mathbb Z}
\def \N { {\rm N}}

\newcommand{\thmref}[1]{Theorem~\ref{#1}}
\newcommand{\lemref}[1]{Lemma~\ref{#1}}

\newcommand{\propref}[1]{Proposition~\ref{#1}}

\numberwithin{equation}{section}

\begin{document}
\title{Palindromic Automorphisms of  Free Nilpotent Groups}
\author[V. G. Bardakov]{Valeriy G. Bardakov}
\author[K. Gongopadhyay]{Krishnendu Gongopadhyay}
\author[M. V. Neshchadim]{Mikhail V. Neshchadim}
\author[M. Singh]{Mahender Singh}

\date{\today}
\address{Sobolev Institute of Mathematics and Novosibirsk State University, Novosibirsk 630090, Russia.}
\address{Laboratory of Quantum Topology, Chelyabinsk State University, Brat'ev Kashirinykh street 129, Chelyabinsk 454001, Russia.}
\email{bardakov@math.nsc.ru}
\address{Indian Institute of Science Education and Research (IISER) Mohali, Sector 81,  S. A. S. Nagar, P. O. Manauli, Punjab 140306, India.}
\email{krishnendu@iisermohali.ac.in}
\address{Sobolev Institute of Mathematics and Novosibirsk State University, Novosibirsk 630090, Russia.}
\email{neshch@math.nsc.ru}
\address{Indian Institute of Science Education and Research (IISER) Mohali, Sector 81,  S. A. S. Nagar, P. O. Manauli, Punjab 140306, India.}
\email{mahender@iisermohali.ac.in}

\subjclass[2010]{Primary 20F28; Secondary 20E36, 20E05}
\keywords{Free nilpotent group; palindromic automorphism; central automorphism, tame automorphism}

\begin{abstract}
In this paper, we initiate the study of palindromic automorphisms of groups that are free in some variety. More specifically, we define palindromic automorphisms of free nilpotent groups and show that the set of such automorphisms is a group. We find a generating set for the group of palindromic automorphisms of free nilpotent groups of step 2 and 3. In particular, we obtain a generating set for the group of  central palindromic automorphisms of these groups. In the end, we determine central palindromic automorphisms of free nilpotent groups of step 3 which satisfy the necessary condition of Bryant-Gupta-Levin-Mochizuki for a central automorphism to be tame.
\end{abstract}
\maketitle

\section{Introduction}
Let $\mathcal M$ be a variety of groups and $F$ be a group that is free in $\mathcal M$. Let  $X=\{x_1,\dots,x_n \}$ be a basis of $F$.  A reduced word $w \eqcirc x_{i_1} x_{i_2} \dots x_{i_m}$ in $X^{\pm1}$ is called a {\it palindrome} in the alphabet $X^{\pm1}$ if it is equal as a word to its reverse word  $\overline{w} \eqcirc x_{i_m} \dots x_{i_2} x_{i_1}$, where by $\eqcirc$ we denote equality of letter by letter.  An element $g\in F$ is called a {\it palindrome} if it can be presented by some word in the alphabet $X^{\pm1}$, which is a palindrome. Note that this definition depends on the generating set $X$ of $F$.

In \cite{C}, Collins defined and investigated palindromic automorphisms of absolutely free groups. Following Collins, we say that an automorphism $\phi$ of $F$ is a {\it palindromic automorphism} if $x_i^{\phi}$ is a palindrome with respect to $X$ for each $1 \leq  i \leq n$. It is not difficult to check that the product of two palindromic automorphisms is again a palindromic automorphism. We will denote the monoid of palindromic automorphisms of $F$ by $\Pi A(F)$.

An automorphism $\phi$ of $F$ of the form
$$\phi: x_i \mapsto \overline{p_i} x_i  p_i~ \textrm{for each}~1 \leq  i \leq n,$$
is called an \emph{elementary palindromic automorphism}. Here $p_1, \dots, p_n \in F$. The sub-monoid of elementary palindromic automorphisms of $F$ is denoted by $E \Pi A(F)$.
We call the group
$$ \Omega S_n = \langle t_i, \alpha_{j,j+1} ~|~\ 1 \leq i \leq n~\textrm{and}~ 1 \leq j \leq n-1 \rangle $$
the \emph{extended symmetric group}, where
$$
t_{i} : \left\{
\begin{array}{ll}
x_i \longmapsto x_i^{-1} &  \\
x_k \longmapsto x_k &  \textrm{for}~k \neq i
\end{array} \right.
$$
and
$$
\alpha_{j, j+1} : \left\{
\begin{array}{ll}
x_j \longmapsto x_{j+1} &  \\
x_{j+1} \longmapsto x_j & \\
x_k \longmapsto x_k & \textrm{for}~ k \neq j.
\end{array} \right.
$$

Clearly, $E\Pi A(F)$ and $ \Omega S_n$ generate $\Pi A (F)$ as a monoid and $\Pi A(F) = E \Pi A(F) \leftthreetimes \Omega S_n$.  Let $IA(F)$ denote the group of those automorphisms of $F$ that induce identity on the abelianization of $F$. Then we have the following short exact sequence
\begin{equation} \label{se1} 1 \to IA (F) \to Aut(F) \to {\rm GL}(n,  \mathbb{Z}) \to 1.\end{equation}
Let $PI(F)=E \Pi A(F) \cap IA(F)$ denote the sub-monoid of \emph{palindromic IA-auto\-mor\-phisms} of $F$.
\medskip

If $F = F_n$ is a free group, then Collins \cite{C} obtained a generating set for $\Pi A (F_n)$. In particular, Collins proved that $E \Pi A(F_n)$ is a group generated by $\mu_{ij}$ for $1 \leq i \not= j \leq n$, where
$$
\mu_{ij} : \left\{
\begin{array}{ll}
x_i \longmapsto x_j x_{i} x_j &  \\
x_k \longmapsto x_k & \textrm{for}~ k \neq i.
\end{array} \right.
$$

In the same paper \cite{C}, Collins conjectured that $E\Pi A(F_n)$ is torsion free for each $n \geq 2$. Using geometric techniques, Glover and Jensen  \cite{gj} proved this conjecture and also calculated the virtual cohomological dimension of $\Pi A(F_n)$. Extending this work in \cite{jmm}, Jensen, McCommand and Meier computed the Euler characteristic of $\Pi A (F_n)$ and $E \Pi A (F_n)$. In \cite{pr}, Piggott and Ruane constructed Markov languages of normal forms for $\Pi A(F_n)$ using methods from logic theory. In \cite{ne1, ne2}, Nekritsukhin investigated some basic group theoretic questions about $\Pi A(F_n)$. In particular, he studied involutions and center of $\Pi A (F_n)$. In a recent  paper \cite{Fullarton}, Fullarton obtained a generating set for the palindromic IA-automorphism  group $PI(F_n)$. This was obtained by constructing an action of $PI(F_n)$ on a  simplicial complex modelled on the complex of partial bases due to Day and Putman \cite{DP}. The papers \cite{gj} and \cite{Fullarton} indicates a deep connection between palindromic automorphisms of free groups and geometry. Recently, Bardakov, Gongopadhyay and Singh \cite{BGS} investigated many algebraic properties of $\Pi A (F_n)$. In particular, they obtained conjugacy classes of involutions in $\Pi A(F_2)$ and investigated residual nilpotency of $\Pi A(F_n)$. They also refined a result of Fullarton \cite{Fullarton} by proving that $PI(F_n)=I A(F_n) \cap E \Pi A'(F_n)$.

The purpose of this paper is to initiate the study of palindromic automorphisms of free groups in varieties of groups. Let $F$ be a free group in some variety. Then $\Pi A(F) = E \Pi A(F) \leftthreetimes \Omega S_n$ and the following problem seems natural.

\begin{problem}\label{problem1}
When is $E \Pi A(F)$ a group? Find a generating set for $E \Pi A(F)$ as a monoid and as a group.
\end{problem}

Let $F^n=F \times \cdots \times F$ ($n$ copies). Then the following problem is connected with  the description of elementary palindromic automorphisms of $F$.

\begin{problem}\label{problem2}
Let $q = (q_1, \dots, q_n) \in F^n$. When does the palindromic map $$\varphi_q : x_i \mapsto \overline{q_i} x_i  q_i~~ \textrm{for}~1 \leq i \leq  n,$$ defines an automorphism of $F$?
\end{problem}

Regarding the monoid of palindromic IA-automorphisms, we pose the following problem.

\begin{problem}\label{problem3}
When is $PI(F)\neq 1$? If  $PI(F) \neq 1$, then find a generating set for $PI(F)$.
\end{problem}

Let $\N_{n,k} = F_n / \gamma_{k+1} F_n$ be the free nilpotent group of rank $n$ and step $k$. In this paper, we investigate the above problems for $\N_{n,k} $ with more precise results for $k=2$ and $3$. The paper is organised as follows.

In  Section 2, we discuss Problem \ref{problem1} and \ref{problem2}. In Theorem \ref{group}, we show that $\Pi A(\N_{n,k})$ is a group. We also obtain some general results regarding central palindromic automorphisms of $\N_{n,k}$ in Theorem \ref{central-palin}.

In  Sections 3 and 4, we discuss Problems \ref{problem1} and \ref{problem3}. We prove that $E \Pi A(\N_{n,1}) \cong E \Pi A(\N_{n,2})$ in Proposition \ref{Nn1=Nn2}. We find a generating set for $\Pi A(\N_{n,2})$ in Proposition \ref{Nn2-gen-set} and prove that $PI(\N_{n,2})=1$ in Proposition \ref{PINn2=1}. Finally, we find a generating set for $E \Pi A(\N_{n,3})$ in Theorem \ref{Nn3-gen-set}.

Note that, the natural homomorphism $F_n \to \N_{n,k}$ induces a homomorphism $$Aut(F_n) \to Aut(\N_{n,k}).$$ It is well known that this homomorphism is not an epimorphism for $k \geq 3$ (see \cite{Andreadakis}). An automorphism of $\N_{n,k}$ is called {\it tame} if it is induced by some automorphism of $F_n$.

\begin{problem}\label{problem5}
Describe tame palindromic automorphisms of $\N_{n,k}$.
\end{problem}

We consider Problem \ref{problem5} in Section 5. In Theorem \ref{central-bgml-condition}, we find a generating set for the group of central palindromic automorphisms of $\N_{n,3}$ which satisfy the necessary condition of Bryant-Gupta-Levin-Mochizuki for such an automorphism to be tame. We also show that some of these automorphisms are tame. Finally, we conclude the paper with some open problems in Section 6.

We use standard notation and convention throughout the paper. All functions are evaluated from left to right. If $G$ is a group, then $G'$ denotes the commutator subgroup of $G$, $Z(G)$ denotes the center of $G$, and $\gamma_n G$ denotes the $n$th term in the lower central series of $G$. Given two elements $g$ and $h$ in $G$, we denote the element $h^{-1}gh$ by $g^h$ and the commutator $g^{-1}h^{-1}gh$ by $[g,h]$. Given $g_1, g_2, \dots, g_k \in G$, we denote the element
$[\cdots [[g_1, g_2], g_3], \dots, g_k]$ by $[g_1,g_2, \dots, g_k]$.

\section{General results on palindromic automorphisms of $\N_{n,k}$}\label{gr}

In this section, we consider free nilpotent groups $\N_{n,k}$ of rank $n$ and step $k$. Let $\{x_1, x_2,\dots, x_n \}$ be a basis for $\N_{n,k}$. Each $n$-tuple $(p_1, \ldots, p_n) \in \N_{n, k}^n$ defines an endomorphism of $N$, which acts on the generators by the rule
$$
x_i^{\psi} = p_i~~\textrm{for each}~ 1 \leq i \leq n.
$$

It is well known that the endomorphism $\psi$ is an automorphism if the matrix $[\psi] = \big(\log_{x_j}(p_i)\big)$ lies in ${\rm GL}(n, \Z)$. The main result of this section is the following theorem.

\begin{theorem}\label{group}
$\Pi A(\N_{n,k})$ is a group.
\end{theorem}

First, we discuss Problem \ref{problem2} for free nilpotent groups.

\begin{prop}\label{p1}
Let $\N_{n, k}$ be the free nilpotent group of rank $n$ and step $k$. Further, let $\varphi$ be the automorphism determined by  $(p_1, p_2, \dots, p_n) \in \N_{n, k}^n$. If $\varphi$ is an elementary palindromic automorphism,  then the matrix $[\varphi]$ is the identity matrix mod $2$.

For $\N_{n, 2}$, the converse is also true. In other words, if the matrix $[\varphi]$ is the identity matrix mod $2$, then automorphism $\varphi$ is an elementary palindromic automorphism.
\end{prop}

\begin{proof}
Suppose that  $\varphi$ is an elementary palindromic automorphism. Then $\varphi = \varphi_q$ for some $q = (q_1, q_2, \dots,q_n) \in \N_{n, k}^n$, where $q_i=x_1^{\alpha_{i_1} }x_2^{\alpha_{i_2}} \ldots x_n^{\alpha_{i_ n}} c_i$ and  $c_i \in \N_{n,k}'$ for each $ 1 \leq i \leq n$.
Then

\begin{eqnarray*}
x_i^{\varphi} &= & \overline{q_i} x_i {q_i} \\
&=&\overline{c_i} x_n^{\alpha_{i_n}} \ldots x_1^{\alpha_{i_1}} x_i  x_1^{\alpha_{i_1}} \ldots x_n^{\alpha_{i_n}} c_i\\
&=& x_1^{2 \alpha_{i_1}} x_2^{2 \alpha_{i_2}} \ldots x_i^{2 \alpha_{i_ i}+1} \ldots  x_n^{2\alpha_{i_n}}   d_i, \hbox{where } d_i \in \N_{n, k}'.
\end{eqnarray*}
Hence $[\varphi] = 2[q] + I$, where $[q]=(\alpha_{i_j})_{i, j=1, \dots, n}$.
This proves the first part of the proposition.

If $\varphi \in Aut (\N_{n, 2})$ and  $[\varphi] \equiv I (\mathrm{mod} ~2)$, then  $$\varphi(x_i)=x_1^{2 \alpha_{i_1}}x_2^{2 \alpha_{i_2}} \ldots x_i^{2 \alpha_{i_i}+1}\ldots  x_n^{2\alpha_{i_n}}   d_i,$$ where $d_i \in \N_{n, 2}'$. Now using normal forms for palindromes in $\N_{n,2}$ (as in  \cite[p. 557]{BG2}), we can reduce the right hand side to the form $\overline{c_i} x_n^{\alpha_{i_n}} \ldots x_1^{\alpha_{i_1}} x_i  x_1^{\alpha_{i_1}} \ldots x_n^{\alpha_{i_n}} c_i$ for some $c_i \in \N_{n, 2}'$. This proves the proposition.
\end{proof}

\begin{corollary}
If the palindromic map $\varphi_q$ defines a palindromic automorphism, then it is an $IA$-automorphism of $\N_{n, k}$ if and only if $[q]=0$.
\end{corollary}

\begin{proof}
Note that the elements $y_1, \dots, y_n$ generate $\N_{n, k}$ if and only if they generate $\N_{n, k}$ modulo the commutator $\N_{n, k}'$. Hence the result follows from the above proposition.
\end{proof}
\medskip

\noindent \textbf{Proof of Theorem \ref{group}}.
Since $\Pi A(\N_{n,k}) = E \Pi A(\N_{n,k}) \leftthreetimes \Omega S_n$, it suffices to prove that $E \Pi A(\N_{n,k})$ is a group. Clearly, $E \Pi A(\N_{n,k})$ is a monoid. It only remains to prove that if $\phi \in E \Pi A(\N_{n,k})$, then $\phi^{-1} \in E \Pi A(\N_{n,k})$. Let $\phi \in E \Pi A(\N_{n,k})$. Then $$x_i^{\phi}= \overline{q_i} x_i  q_i~ \textrm{for}~q_i \in \N_{n,k}~\textrm{for each}~1 \leq i \leq  n.$$
Define $\phi_1= \phi$ and $q_{i_1}=q_i$ for each $1 \leq i \leq  n$. Then we have 
$$q_{i_1}= x_1^{\alpha_{i_1}} \cdots x_n^{\alpha_{i_n}} (\mathrm{mod} ~\gamma_2 \N_{n,k})$$
for some $\alpha_{i_j} \in \mathbb{Z}$. As discussed in Proposition \ref{p1}, we have $[\varphi] = 2(\alpha_{i_j}) + I \in {\rm GL}(n, \mathbb{Z})$ and hence its inverse $[\varphi]^{-1} = 2(\beta_{i_j}) + I$ for some $\beta{i_j} \in \mathbb{Z}$.

Define $\psi_1 \in E \Pi A(\N_{n,k})$ by the rule
$$x_i^{\psi_1}=  x_n^{\beta_{i_n}} \cdots x_1^{\beta_{i_1}} x_i x_1^{\beta_{i_1}} \cdots x_n^{\beta_{i_n}}~\textrm{for each}~1 \leq i \leq  n.$$
Let $\phi_2= \phi_1 \psi_1$. Then we have
$$x_i^{\phi_2}=  \overline{q_{i_2}} x_i q_{i_2} ~\textrm{for each}~1 \leq i \leq  n$$
and for some $q_{i_2} \in \gamma_2 \N_{n,k}$.

Define $\psi_2 \in E \Pi A(\N_{n,k})$ by the rule
$$x_i^{\psi_2}=  \overline{q_{i_2}^{-1}} x_i q_{i_2}^{-1} ~\textrm{for each}~1 \leq i \leq  n.$$
It follows that
$$x_i^{\phi_2 \psi_2}=( \overline{q}_{i_2} x_i q_{i_2})^{\psi_2}= \overline{q_{i_2}^{\psi_2}} ~ \overline{q_{i_2}^{-1}} x_i q_{i_2}^{-1}  q_{i_2}^{\psi_2} ~\textrm{for each}~1 \leq i \leq  n.$$
Since $\psi_2$ is trivial modulo $\gamma_2 \N_{n,k}$, we have 
$$q_{i_2}^{\psi_2}=q_{i_2} (\mathrm{mod} ~\gamma_3 \N_{n,k}).$$
Hence $q_{i_3}:=q_{i_2}^{-1}q_{i_2}^{\psi_2} \in  \gamma_3 \N_{n,k}$. Let $\phi_3= \phi_2 \psi_2$. Then we have
$$x_i^{\phi_3}=  \overline{q_{i_3}} x_i q_{i_3} ~\textrm{for each}~1 \leq i \leq  n.$$
Define $\psi_3 \in E \Pi A(\N_{n,k})$ by the rule
$$x_i^{\psi_3}=  \overline{q_{i_3}^{-1}} x_i q_{i_3}^{-1} ~\textrm{for each}~1 \leq i \leq  n.$$
Then we have
$$x_i^{\phi_3 \psi_3}=( \overline{q}_{i_3} x_i q_{i_3})^{\psi_3}= \overline{q_{i_3}^{\psi_3}} ~ \overline{q_{i_3}^{-1}} x_i q_{i_3}^{-1}  q_{i_3}^{\psi_3} ~\textrm{for each}~1 \leq i \leq  n.$$
Since $\psi_3$ is trivial modulo $\gamma_2 \N_{n,k}$, we have 
$$q_{i_3}^{\psi_3}=q_{i_3} (\mathrm{mod} ~\gamma_4 \N_{n,k}).$$
Hence $q_{i_4}:=q_{i_3}^{-1}q_{i_3}^{\psi_3} \in  \gamma_4 \N_{n,k}$. Let $\phi_4= \phi_3 \psi_3$. Then we have
$$x_i^{\phi_4}=  \overline{q_{i_4}} x_i q_{i_4} ~\textrm{for each}~1 \leq i \leq  n.$$
Continuing like this, at the $(k+1)$st step, we obtain
$$x_i^{\phi_{k+1}}=  \overline{q_{i_{k+1}}} x_i q_{i_{k+1}} ~\textrm{for each}~1 \leq i \leq  n,$$
where $q_{i_{k+1}} \in \gamma_{k+1} \N_{n,k}=1$. This implies $\phi_{k+1}=1$ and hence $\psi_{k+1}=1$. Thus we obtain two sequences of automorphisms $\psi_1, \psi_2, \dots, \psi_k$ and $\phi_1, \phi_2, \dots, \phi_k$ in $E \Pi A(\N_{n,k})$ such that $\phi_{l+1}=\phi_l \psi_l$ for each $l \geq 1$. This gives 
$$1= \phi_{k+1}=\phi_k \psi_k=\phi_{k-1} \psi_{k-1}\psi_k=\cdots = \phi_1 \psi_1 \psi_2 \cdots \psi_k.$$
Hence $\phi^{-1}=\phi_1^{-1}= \psi_1~ \psi_2\cdots \psi_k \in E \Pi A(\N_{n,k})$. This proves the theorem. $\Box$
\medskip

We define the set of linear palindromic automorphisms by
$$L \Pi(\N_{n,k})=\bigg \{ \phi \in E \Pi A(\N_{n,k}) \ | \ \phi:  \left.
\begin{array}{l}
x_i \longmapsto x_n^{\alpha_n}\ldots x_1^{\alpha_1} x_i x_1^{\alpha_1} \ldots x_n^{\alpha_n} \\
x_j \longmapsto x_j  ~\mbox{for}~j \not= i
\end{array} \right. \bigg\}.$$

As a consequence of the above theorem, we obtain the following corollary.

\begin{corollary}
Every elementary palindromic automorphism $\phi$ of $\N_{n,k}$ can be written as $\phi=\varphi \psi$ for some $\varphi \in L \Pi (N_{n,k}) $ and $\psi \in  PI(\N_{n,k})$.
\end{corollary}

We know that in $\N_{n,2}$  the reverse of the basis commutator of weight two is equal to its inverse, that is,
$$
\overline{[x_i, x_j]} = [x_i, x_j]^{-1}.
$$

The following lemma is a generalization of this observation for step  $k$ nilpotent groups  for $k \geq 2$.

\begin{lemma} \label{l:1}
Let $y_1,\ldots,y_k \in \left\{x_1,\ldots,x_n \right\}$. Then the following equality holds in $\N_{n,k}$
$$
\overline{[y_1,\ldots,y_k]}=[y_1,\ldots,y_k]^{(-1)^{k+1}}.
$$
\end{lemma}

\begin{proof}
We use induction on the step nilpotence  $k$.
Since
$$
[a,b]=a^{-1}b^{-1}ab, \,\,\, [b,a]=[a,b]^{-1},
$$
we have
$$
\overline{[a,b]}=bab^{-1}a^{-1}=[b^{-1},a^{-1}].
$$
In a nilpotent group of step $k$, the following holds
$$
[z_1^{\alpha_1},\ldots,z_k^{\alpha_k}]=[z_1,\ldots,z_k]^{\alpha_1\cdot \ldots \cdot \alpha_k}
$$
for all $z_1,\ldots,z_k \in \N_{n,k}$ and all $\alpha_1, \ldots , \alpha_k \in \mathbb{Z}$.

For $k=2$, we have
$$
\overline{[y_1,y_2]}=[y_2^{-1},y_1^{-1}]=[y_2,y_1]=[y_1,y_2]^{-1}.
$$
Suppose that our assertion is true for $k>2$. Then for  $k+1$, we have

\begin{eqnarray}
 \overline{[y_1,\ldots,y_k, y_{k+1}]} & = & [y_{k+1}^{-1},\overline{[y_1,\ldots,y_k]}^{-1}] \nonumber\\
& = &  [y_{k+1},\overline{[y_1,\ldots,y_k]}]\nonumber\\
& =  & [\overline{[y_1,\ldots,y_k]},y_{k+1}]^{-1} \nonumber\\
& =  & [[y_1,\ldots,y_k]^{(-1)^{k+1}}, y_{k+1}]^{-1}\nonumber\\
& = & [y_1,\ldots,y_k, y_{k+1}]^{(-1)^{k+2}} \nonumber.
\end{eqnarray}
This proves the lemma.
\end{proof}

\medskip

Recall that an automorphism of a group is called {\it normal} if it sends each normal subgroup onto itself and it is called {\it central} if it induces identity on the central quotient. We prove the following result.

\begin{theorem}\label{central-palin}
Let $\N_{n,k}$ be the free nilpotent group of rank $n$ and step $k$.
\begin{enumerate}
\item If $k$ is even, then the group of central palindromic automorphisms of $\N_{n,k}$ is trivial.
\item If $k$ is odd, then the group of central palindromic automorphisms of $\N_{n,k}$ is non-trivial and has a non-trivial intersection with the group of normal automorphisms of $\N_{n,k}$.
\end{enumerate}
\end{theorem}

\begin{proof}
First, suppose that $k$ is even. Let $\varphi$ be a central palindromic automorphism of $\N_{n,k}$. Then
$$
x_i^{\varphi}=x_i c_i,$$
where $c_i \in \gamma_k \N_{n,k}$ for each $1 \leq i \leq n$. Since $x_i c_i$ is a palindrome, it follows that $ \overline{x_i c_i}=x_i c_i$ for each $1 \leq i \leq n.$ By  \lemref{l:1}, $\overline{c_i}=c_i^{-1}$.
Hence,
$$
\overline{x_i c_i} = \overline{c_i} x_i = c_i^{-1} x_i = x_i c_i^{-1}=x_i c_i,
$$
which gives $c_i=1$ for all $1 \leq i \leq n$. Hence $\varphi$ is trivial.

Next, we consider the case of odd $k$. Fix some elements $c_1,\ldots,c_n$ in the center $Z(\N_{n,k}) = \gamma_k \N_{n,k}$. Define an automorphism $\varphi$ of $\N_{n,k}$ by the rule
$$
x_i^{\varphi}=x_i c_i^2 ~\textrm{for each}~ 1 \leq i\leq n.
$$
By  \lemref{l:1}, $\varphi$ is a palindromic automorphism. Indeed,  $c_i$ can be presented as a product of even powers commutators of weight  $k$
$$
[y_1,\ldots,y_k]^{2\alpha},\,\,\,\mbox{where} \,\,\, y_1,\ldots,y_k \in \left\{x_1,\ldots,x_n \right\}~\textrm{and}~\alpha \in \mathbb{Z}.
$$
But the words
$$
[y_1,\ldots,y_k]^{\alpha} x_i [y_1,...,y_k]^{\alpha}
$$
are palindromes. Hence, if $k$ is odd, then the subgroup of central palindromic automorphisms of $\N_{n,k}$ is non-trivial.

Define the central automorphism
$$
g^{\psi}=g [g,z_1,z_2,\ldots,z_{k-1}]~\textrm{for}~ g \in N_{n,k},
$$
where $z_1,z_2,\ldots,z_{k-1} \in \N_{n,k}$ are some fixed elements. By \cite{L} (see also \cite{End}), $\psi$ is a normal automorphism. If we put $z_1=z_2= \cdots =z_{k-2}=x_1$ and $z_{k-1}=x_1^2$ in the above formula, then we obtain
$$
x_i^{\psi}=x_i [x_i,x_1,x_1, \ldots, x_1]^2~\textrm{for}~1 \leq i \leq n.
$$
By the previous discussion, $\psi$ is a palindromic automorphism. This proves the claim.
\end{proof}
\medskip

\begin{remark}
If $w\in \gamma_{s}F_n$, then $\overline{w}\in \gamma_{s}F_n$. Indeed, the map $w \mapsto \overline{w}^{-1}$ is an automorphism of $F_n$, which acts on the generators $x_1,\ldots,x_n$ as inversion $x_i \mapsto x_i^{-1}$. Hence, if $w\in \gamma_{s} F_n$, then $\overline{w}^{-1}\in \gamma_{s} F_n$ and $\overline{w}\in \gamma_{s} F_n$.
\end{remark}

We conclude this section by giving a formula for the product of commutators
$$
[y_1,\ldots,y_{2k}] \overline{[y_1,\ldots,y_{2k}]}
$$
in $\N_{n,2k+1}$. By \lemref{l:1}, we have
$$
[y_1,\ldots,y_{2k}] \overline{[y_1,\ldots,y_{2k}]}\equiv 1 (\mathrm{mod} ~\gamma_{2k+1} \N_{n,2k+1}).
$$
In other words, this is a central element in  $\N_{n,2k+1}$. Define a sequence of words  $w_{2k}\in \gamma_{2k+1} F_n$ for $k = 1, 2, \ldots$ by
the recursive formula
$$
w_2=[y_1,y_2, y_1 y_2]
$$
and
$$
w_{2k+2}=[w_{2k}, y_{2k+1}, y_{2k+2}][y_1,\ldots,y_{2k}, y_{2k+1}, y_{2k+1}y_{2k+2},y_{2k+2}] ~\textrm{for}~ k=1,2,\ldots.
$$

With these notations, we prove the following proposition which is of independent interest.

\begin{prop} \label{l:4}
The following holds in $\N_{n,2k+1}$
$$
[y_1,\ldots,y_{2k}] \overline{[y_1,\ldots,y_{2k}]}= w_{2k},
$$
where $y_1,\ldots,y_{2k}$ are arbitrary elements in  $\N_{n,2k+1}$.
\end{prop}

\begin{proof} We use induction on  $k$. For $k=1$, the assertion was proven in \lemref{l:3}. Suppose that the lemma is true for $k$. Then we prove it for $k+1$.
We shall use the commutator identities
$$
[a^{-1},b]=[b,a][b,a,a^{-1}]~\textrm{and}~ [a,b^{-1}]=[b,a][b,a,b^{-1}].
$$
Denote by
$$
z_{2k}=[y_1,\ldots,y_{2k}].
$$
We will do all the following calculations in $\N_{n,2k+3}$, that is, modulo  $\gamma_{2k+4}F_n$. We have

\begin{eqnarray}
z_{2k+2} \overline{z_{2k+2}} & = & z_{2k+2} \overline{[z_{2k},y_{2k+1},y_{2k+2}]} \nonumber\\
& = & z_{2k+2} [y_{2k+2}^{-1},\overline{[z_{2k},y_{2k+1}]}^{-1}] \nonumber\\
& = & z_{2k+2} [y_{2k+2}^{-1},[y_{2k+1}^{-1},\overline{z_{2k}}^{-1}]^{-1}] \nonumber\\
&  & \textrm{by~ the~ induction~ hypothesis~ we ~get} \nonumber\\
& = & z_{2k+2} ~ [y_{2k+2}^{-1},[y_{2k+1}^{-1},z_{2k}w_{2k}^{-1}]^{-1}] \nonumber\\
& = & z_{2k+2} ~ [y_{2k+2}^{-1},[z_{2k}w_{2k}^{-1},y_{2k+1}^{-1}]] \nonumber\\
& = & z_{2k+2} ~ [y_{2k+2}^{-1},[z_{2k},y_{2k+1}^{-1}]]~   [y_{2k+2}^{-1},[w_{2k}^{-1},y_{2k+1}^{-1}]] \nonumber\\
& = & z_{2k+2} ~ [y_{2k+2}^{-1},[z_{2k},y_{2k+1}^{-1}]]~   [w_{2k},y_{2k+1},y_{2k+2}]  \nonumber\\
& = & z_{2k+2} ~ [y_{2k+2}^{-1},[y_{2k+1},z_{2k}][y_{2k+1}, z_{2k}, y_{2k+1}^{-1}]] ~  [w_{2k}, y_{2k+1}, y_{2k+2}]\nonumber\\
& = & z_{2k+2} ~  [y_{2k+2}^{-1}, [y_{2k+1},z_{2k}]]~  [y_{2k+2}^{-1}, [y_{2k+1},z_{2k},y_{2k+1}^{-1}]]~   [w_{2k},y_{2k+1},y_{2k+2}] \nonumber\\
& = & z_{2k+2} ~  [y_{2k+1}, z_{2k},y_{2k+2}]~  [y_{2k+1},z_{2k}, y_{2k+2},y_{2k+2}^{-1}] \nonumber\\
& & ~ [z_{2k},y_{2k+1},y_{2k+1},y_{2k+2}] ~  [w_{2k},y_{2k+1},y_{2k+2}] \nonumber\\
& = & z_{2k+2} ~   [[z_{2k},y_{2k+1}]^{-1}, y_{2k+2}]~  [z_{2k},y_{2k+1}, y_{2k+2},y_{2k+2}] \nonumber\\
& & [z_{2k},y_{2k+1},y_{2k+1},y_{2k+2}] ~    [w_{2k},y_{2k+1},y_{2k+2}] \nonumber\\
& = & z_{2k+2} ~  [[z_{2k},y_{2k+1}]^{-1}, y_{2k+2}]~  w_{2k+2}\nonumber\\
& = & z_{2k+2} ~  [y_{2k+2},[z_{2k},y_{2k+1}]]~  [y_{2k+2},[z_{2k},y_{2k+1}], [z_{2k},y_{2k+1}]^{-1}]~  w_{2k+2}\nonumber\\
& = & z_{2k+2} ~  [z_{2k},y_{2k+1},y_{2k+2}]^{-1}~  w_{2k+2} \nonumber\\
& = &  z_{2k+2} ~  z_{2k+2}^{-1}~  w_{2k+2} \nonumber\\
& = &   w_{2k+2} \nonumber.
\end{eqnarray}
This proves the proposition.
\end{proof}

\section{Elementary palindromic automorphisms of $\N_{n,2}$}

For $k=1$, the group $\N_{n,1}$ is a free abelian group of rank $n$. Therefore, $Aut(\N_{n,1})$ can be identified with ${\rm GL}(n, \Z)$. Note that $\Pi A (\N_{n,1})=E \Pi A(\N_{n,1}) \rtimes \Omega S_n$. It follows from \propref{p1} that an automorphism $\phi \in E\Pi A(\N_{n,1})$ if  $[\phi]$ is an element of the kernel of the natural epimorphism ${\rm SL}(n,\Z) \to {\rm SL}(n, \Z/2 \Z)$. But, for example, the matrix $diag(-1, -1, 1) \in {\rm SL}(3,\Z)$ does not define an elementary palindromic automorphism of $\N_{n,1}$.

\begin{prop}\label{Nn1=Nn2}
$E\Pi A(\N_{n,2})$ is isomorphic to $E \Pi A(\N_{n,1})$.
\end{prop}
\begin{proof}
Any element of $\N_{n,2}$ is of the form
$$p=x_1^{\alpha_1} \dots x_n^{\alpha_n} \prod_{1\leq l <k \leq n} z_{kl}^{\beta_{kl}}, \hbox{ where } z_{kl}=[x_k, x_l].$$
We see that,    $\overline {z_{kl}} =x_l x_k x_l^{-1} x_k^{-1} = [x_l^{-1}, x_k^{-1}]$. Since $[x_l, x_k]$ is a central element in $\N_{n,2}$, we have
$$[x_l^{-1}, x_k^{-1}]=x_l x_k x_l^{-1} x_k^{-1}=x_l x_k [x_l, x_k] x_k^{-1} x_l^{-1} =[x_l, x_k].$$
This implies $\overline {z_{kl}}=[x_l, x_k]=[x_k, x_l]^{-1} = z_{kl}^{-1}$ and hence
$$
\bar u= \prod_{1 \leq l <k \leq n}  {z}_{kl}^{-\beta_{kl}} \prod_{j=0}^{n-1} x_{n-j}^{\alpha_{n-j}}.
$$
Thus any automorphism $\phi: x_i \mapsto \overline p x_i p$ for $1 \leq i \leq n$ is of the form $$x_i^{\phi}=x_1^{2 \alpha_1} \dots   x_i^{2 \alpha_i +1} \dots x_n^{2 \alpha_n},$$ which is an element of $E \Pi A(\N_{n, 1})$. Hence the result follows.
\end{proof}

In \cite[Proposition 2.7]{BG2}, some normal form of palindromes in $\N_{n, 2}$ was obtained. Using this normal form and a result of Collins on generators of $\Pi A(F_n)$, we easily obtain the following.

\begin{prop}\label{Nn2-gen-set}
$\Pi A (\N_{n, 2})$ is generated by $t_i, ~ \alpha_{l, l+1}$ and
$$
\mu_{ij} : \left\{
\begin{array}{ll}
x_i \longmapsto x_jx_ix_j & \\
x_k \longmapsto x_k & \textrm{for}~k \neq i,
\end{array} \right.
$$
where $1 \leq i \neq j \leq  n$ and  $1 \leq l\leq n-1$.
\end{prop}

Next, we consider the palindromic IA-group.

\begin{prop}\label{PINn2=1}
$PI(\N_{n,2})=E\Pi A(\N_{n, 2}) \cap IA(\N_{n, 2})=1.$
\end{prop}

\begin{proof}
If $\phi \in E \Pi A(\N_{n, 2})$, then it is generated by the elementary palindromic automorphisms $\mu_{ij}$ given above. Using the normal forms for $\N_{n, 2}$ (see \cite{BG2}),  we see that $\mu_{ij}$ does not belong to $IA(\N_{n, 2})$ for all $i, j$. Hence the result follows.

Alternatively, note that $\phi \in IA(\N_{n, 2})$ implies $[\phi]=0$. So, $x_i^{\phi}=\overline{c_i} x_i c_i$ for some $c_i \in \N_{n, 2}'$. But every element in $\N_{n, 2}'$ can be written as product of $z_{ij}=[x_i, x_j]$ for $1 \leq i, j \leq n$. Noting that $\overline{z_{ij}}=z_{ij}^{-1}$ and $z_{ij}$ commute with $x_j$'s,  the result follows.
\end{proof}

\begin{remark}
The above proposition is not true in general for $\N_{n, k}$ for $k \geq 3$. For example, consider the following automorphism $\phi$ in $Aut(\N_{2, 3})$
$$
\phi : \left\{
\begin{array}{ll}
x_1 \longmapsto x_1[x_2, x_1, x_1]^2 = [x_2, x_1, x_1] x_1 \overline{[x_2, x_1, x_1]}&  \\
x_2 \longmapsto x_2.
\end{array} \right.
$$
Note that $[y, x, x]$ is a central element in $\N_{2, 3}$ and $\overline{[y, x, x]}=[y, x, x]$. Hence, $\phi$ is a non-trivial element of $PI(\N_{2, 3})$.
\end{remark}
 
\section{Elementary palindromic automorphisms of $\N_{n,3}$}

In this section, we find a generating set for $E \Pi A(\N_{n,3})$. It is evident that this generating set contains automorphism $\mu_{ij}$ for $1 \leq i \not= j \leq n$ and some central automorphisms of $\N_{n,3}$. Note that this claim is true for arbitrary free nilpotent groups $\N_{n,k}$. Hence, we need to study central palindromic automorphisms. We remark that every central palindromic automorphism is an elementary palindromic automorphism.

First, we prove the following lemma.

\begin{lemma} \label{l:2}
If $w x_1 \overline{w} x_1^{-1}\equiv 1 (\mathrm{mod} ~\gamma_{3}\N_{n,3})$,
then $w \equiv 1 (\mathrm{mod} ~ \gamma_{2}\N_{n,3})$.
\end{lemma}

\begin{proof}
Write the word $w$ as a product
$$
w=w_1 w_2,
$$
where $w_1 = x_1^{\alpha_1} x_2^{\alpha_2} \ldots x_n^{\alpha_n}$, $\alpha_i \in \mathbb{Z}$ and $w_2 \in \gamma_{2} \N_{n,3}$.
Then
$$
w_1 w_2 x_1 \overline{w_2}~ \overline{w_1} x_1^{-1}\equiv 1 \, (\mathrm{mod} ~ \gamma_{3}\N_{n,3}).
$$
In particular,
$$
w_1 x_1 \overline{w_1} x_1^{-1}\equiv 1 \, (\mathrm{mod} ~\gamma_{2}\N_{n,3}).
$$
In other words,
$$
w_1 \overline{w_1} \equiv 1 \, (\mathrm{mod} ~\gamma_{2}\N_{n,3}).
$$
Since $w_1 \equiv \overline{w_1} \, (\mathrm{mod} ~\gamma_{2}\N_{n,3})$, it follows that $w_1^2 \equiv 1 \, (\mathrm{mod} ~\gamma_{2}\N_{n,3})$, that is,
$w_1 \equiv 1 \, (\mathrm{mod}~ \gamma_{2}\N_{n,3})$.
\end{proof}

Note that a central palindromic automorphism $\varphi$ acts on the generators  $x_1,\ldots,x_n$ by the rule
$$
x_i^{\varphi}=w_i x_i \overline{w_i},~\textrm{where}~w_i \in \N_{n,3}~\textrm{and}~~ 1 \leq i\leq n.
$$
Then by \lemref{l:2}, we can assume that $w_i=u_i v_i$, where $u_i$ is a product of commutators of weight 2 and  $v_i \in  \gamma_3 N_{n,3}$. By \lemref{l:1}, $\overline{v_i}=v_i$ and hence
$$
x_i^{\varphi}=u_i x_i \overline{u_i} v_i^2~\textrm{for}~ 1 \leq i \leq n.
$$

Next, we prove the following lemma.

\begin{lemma} \label{l:3}
If $x \in \N_{n,3}$ and
$$
 u= \prod\limits_{1\leq b < a \leq n} [x_a, x_b]^{p_{ab}},~\textrm{where}~ p_{ab}\in \mathbb{Z},
$$
then
$$
 u x \overline{u}=
        x \prod\limits_{1\leq b < a \leq n}
        \left([x_a, x_b,x][x_a, x_b,x_b] [x_a,x_b,x_a]\right)^{p_{ab}}.
$$
\end{lemma}

\begin{proof}
We have
$$
 u x \overline{u}=xx^{-1}uxu^{-1}u\overline{u}=
 x[x,u^{-1}]u\overline{u}=
  x[u,x]u\overline{u}.
$$
Note that $[u,x]\in  \gamma_3 \N_{n,3}$ and $u\overline{u}\in  \gamma_3 \N_{n,3}$. Since
$$
 u= \prod\limits_{1\leq b < a \leq n} [x_a, x_b]^{p_{ab}}~\textrm{where}~ p_{ab}\in \mathbb{Z},
$$
it follows that
$$
 [u,x]= \prod\limits_{1\leq b < a \leq n} [x_a, x_b,x]^{p_{ab}}
$$
and
$$
 u\overline{u}= \prod\limits_{1\leq b < a \leq n} \left([x_a, x_b] [x_b^{-1},x_a^{-1}]\right)^{p_{ab}}.
$$
Further,
\begin{eqnarray}
 [x_a, x_b] [x_b^{-1},x_a^{-1}] & = & [x_a, x_b](x_bx_a)[x_b,x_a](x_bx_a)^{-1} \nonumber\\
& = & [x_a, x_b](x_bx_a)[x_a,x_b]^{-1}(x_bx_a)^{-1} \nonumber\\
& = & [[x_a,x_b]^{-1},(x_bx_a)^{-1}]\nonumber\\
& = & [x_a,x_b,x_b x_a] \nonumber.
\end{eqnarray}

Hence,
$$
 u\overline{u}= \prod\limits_{1\leq b < a \leq n} \left([x_a, x_b,x_b] [x_a,x_b,x_a]\right)^{p_{ab}}.
$$
In this way, we have
$$
 u x \overline{u}=
        x \prod\limits_{1\leq b < a \leq n}
        \left([x_a, x_b,x][x_a, x_b,x_b] [x_a,x_b,x_a]\right)^{p_{ab}}.
        $$
\end{proof}
\medskip

Now, if
$$
 u_i= \prod\limits_{1\leq b < a \leq n} [x_a, x_b]^{p_{ab,i}},~\textrm{where}~ p_{ab,i}\in \mathbb{Z}~\textrm{and}~ 1 \leq i \leq n,
$$
then
\begin{eqnarray}
 x_i^{\varphi} & = & (u_i x_i \overline{u_i}) v_i^2 \nonumber\\
& = & x_i \left( \prod\limits_{1\leq b < a \leq n} \left([x_a, x_b,x_i][x_a, x_b,x_b] [x_a,x_b,x_a]\right)^{p_{ab,i}} \right) v_i^2 \nonumber
\end{eqnarray}
for $1 \leq i  \leq n$. Since $\varphi$ acts independently on each letter  $x_i$ and all the letters are equivalent, it is enough  to understand what is the action of $\varphi$ on the letter  $x_n$. From the above formulas it follows that  $\varphi$ is a product of automorphisms of the type
$$
 x_n^{\varphi_{ab}}= x_n [x_a, x_b,x_n][x_a, x_b,x_b] [x_a,x_b,x_a]
$$
and
$$
 x_n^{\varphi_{abc}}= x_n [x_a, x_b,x_c]^2.
$$
We can assume that  $c\geq b$ and $a > b$. Hence, these formulas contains only basis commutators \cite[Chapter 5]{Magnus}.

Note that the words
$$
 w_{ab} =  [x_a, x_b,x_n][x_a, x_b,x_b] [x_a,x_b,x_a],~\textrm{where}~ 1\leq b < a \leq n
$$
are independent in  $\gamma_3 \N_{n,3}$ and the number of these words is equal to  $n(n-1)/2$. In other words, the number of these words is equal to the dimension of the quotient  $\gamma_2 \N_{n,3} /\gamma_3 \N_{n,3}$.
Define the subgroup
$$
H=\left\langle w_{ab},\,\, [x_a, x_b,x_c]^2 ~|~  1\leq b < a \leq n~\textrm{and}~ 1 \leq b \leq c \leq n   \right \rangle.
$$
Then we have the following isomorphism
$$
\gamma_3 \N_{n,3} / H \simeq \mathbb{Z}_2^q,
$$
where $q=\dim (\gamma_3 \N_{n,3}) - \dim (\gamma_2 \N_{n,3}/\gamma_3 \N_{n,3})=
\displaystyle\frac{n(n^2-1)}{3}-\frac{n(n-1)}{2}$.

Thus, we have proved the following.

\begin{prop}\label{cgen}
The group of central palindromic automorphisms of $\N_{n,3}$ is generated by the automorphisms  $\varphi_{ab,i}$ and $\varphi_{abc,i}$, where
$1\leq b < a \leq n$, $1 \leq b \leq c \leq n$ and $1 \leq i \leq n$. Further, these act on the generators  $x_1,\ldots, x_n$ in the following manner
$$
\varphi_{ab,i} : \left\{
\begin{array}{l}
x_i \longmapsto x_i [x_a, x_b,x_i][x_a, x_b,x_b] [x_a,x_b,x_a] \\
x_j \longmapsto x_j  ~\mbox{for}~j \not= i
\end{array} \right.
$$
and
$$
\varphi_{abc,i} : \left\{
\begin{array}{l}
x_i \longmapsto x_i [x_a, x_b,x_c]^2 \\
x_j \longmapsto x_j  ~\mbox{for}~j \not= i.
\end{array} \right.
$$

The quotient of the group of central automorphisms by the subgroup of central palindromic automorphisms is isomorphic to the group $\mathbb{Z}_2^{nq}$, where $$q=\dim (\gamma_3 \N_{n,3}) - \dim (\gamma_2 \N_{n,3}/\gamma_3 \N_{n,3})=
\displaystyle\frac{n(n^2-1)}{3} - \frac{n(n-1)}{2}.$$
\end{prop}

From this follows the main result of this section.

\begin{theorem}\label{Nn3-gen-set}
The group $E \Pi A(\N_{n,3})$ is generated by automorphisms $\mu_{ij}$, where $1 \leq i \not= j \leq n$ and by central automorphisms
$\varphi_{ab,i}$ and $\varphi_{abc,i}$, where $1\leq b < a \leq n$, $1 \leq b \leq c \leq n$ and $1 \leq i \leq n$.
\end{theorem}

\section{Tame palindromic automorphisms of $\N_{n,3}$}
Recall that an automorphism of  $\N_{n,k}=F_n/ \gamma_{k+1} F_n$ is called {\it tame} if it is induced by some automorphism of free group  $F_n$. In the opposite case, it is called a \emph{wild} automorphism.

In the paper \cite{Bry}, a necessary condition for a central automorphism of a free nilpotent group  $\N_{n,k}$ with $k,n \geq 2$ to be a tame automorphism was determined. The purpose of this section is to describe central palindromic automorphisms of free nilpotent groups $\N_{n,3}$ for which this necessary condition holds.

To formulate the necessary condition, we recall the definition of Fox's derivatives. See \cite{Fox} or \cite[Chapter 3]{Bir} for details. Let  $\mathbb{Z}F_n$ be the integral group ring of the free group $F_n$. The $j$-th Fox derivative is a map
$$
\partial_j : \mathbb{Z}F_n \to \mathbb{Z}F_n
$$
defined on the generators $x_1, \ldots, x_n$ by the rule
$$
\partial_j(x_i)=
\left\{
\begin{array}{cc}
  1 & i=j \\
  0 & i\neq j. \\
\end{array}
\right.
$$
Each $\partial_j$ is a $\mathbb{Z}$-linear map and the following condition holds
$$
\partial_j(uv)=\partial_j(u)+u\partial_j(v)
$$
for all $u,v \in F_n$.
From this, it follows that
$$
\partial_j(u^{-1})=-u^{-1}\partial_j(u)~\textrm{for all}~ u \in F_n.
$$

The ring $\mathbb{Z}F_n$ has a fundamental ideal $\Delta$ called the augmentation ideal given by
$$
\Delta= \mathrm{Ker} \left( \varepsilon : \mathbb{Z}F_n \rightarrow \mathbb{Z}\right).
$$
Here $\varepsilon$ is a ring homomorphism defined on the generators $x_1, \ldots, x_n$ as
$$
\varepsilon(x_i)= 1~\textrm{for all}~1 \leq i \leq n.
$$
From the evident relations
$$
uv-1=u(v-1)+(u-1)
$$
and
$$[u,v]-1=u^{-1}v^{-1}((u-1)(v-1)-(v-1)(u-1))$$
where $u,v \in F_n$ and $[u,v]=u^{-1}v^{-1}uv$, it follows that the ideal  $\Delta$ is generated by elements $x_1-1, \ldots, x_n-1$. Further, it follows that $w-1 \in \Delta^{k}$ for all  $w \in \gamma_{k} F_n$. It also follows that $[\Delta,\Delta]$ is an ideal of $\mathbb{Z}F_n$ and is generated by ring commutators $ab-ba$ for $a,b \in \Delta$. Some other properties of  $\Delta$ and connections with Fox's derivatives one can found in \cite{Gup, Pas}.

A necessary condition for a central automorphism of a free nilpotent group  $\N_{n,k}$ to be a tame automorphism is given by the following theorem of Bryant-Gupta-Levin-Mochizuki \cite{Bry}.

\begin{theorem}  \cite{Bry}.\label{t3.1}
 Let $n$ and $k$ be positive integers, where $k\geq 2$. Let $w_1, \ldots, w_n$ be elements of $\gamma_{k} \N_{n,k}$
and $\varphi$ be the automorphism of $\N_{n,k}$ satisfying $x_i^{\varphi}=x_i w_i$ for each $1 \leq i \leq n$.
Let $u_1, \ldots, u_n$ be elements of $\gamma_{k} F_n$ such that $u_i \gamma_{k+1} F_n=w_i$.
If $\varphi$ is tame, then
$$
\partial_1 u_1 + \cdots + \partial_n u_n \in \left( \Delta^{k-1}\cap [\Delta,\Delta]\right) +\Delta^k.
$$
\end{theorem}

It can be seen that, if a central automorphism $\varphi$ of $\N_{n,3}$ given by
$$
x_i^{\varphi}=x_i w_i~\textrm{for}~w_i \in \gamma_{3} \N_{n,3}~\textrm{and}~1 \leq i \leq n
$$
is tame, then
$$
\partial_1 w_1+ \cdots +\partial_n w_n \equiv 0 \left( \mathrm{mod} ~ R \right),
\eqno{(1)}
$$
where $R=[\Delta,\Delta]+\Delta^3$. This follows from the fact that $\Delta^2 \subseteq [\Delta,\Delta]$. Since the ideal  $[\Delta,\Delta]$ is generated by the ring commutators $\big((u-1)(v-1)-(v-1)(u-1)\big)=uv-vu$, it follows that the quotient ring  $\mathbb{Z}F_n/R$
is commutative. Moreover, from the equality
$$
w(u-1)(v-1)=(u-1)(v-1)+(w-1)(u-1)(v-1),
$$
we get
$$
x_i(x_j-1)(x_l-1)\equiv (x_j-1)(x_l-1) \left( \mathrm{mod}~ R \right),~\textrm{where}~ i,j,l \in \{ 1, \ldots, n \}.
$$

From the above properties of the quotient ring  $\mathbb{Z}F_n/R$ and Fox's derivatives, we obtain
$$
\begin{array}{l}
  \partial_i [x_a,x_b,x_c]\equiv0, \\
  \partial_i [x_a,x_b,x_i]\equiv 0, \\
  \partial_i [x_i,x_a,x_b]\equiv (x_a-1)(x_b-1), \\
  \partial_i [x_a,x_i,x_b]\equiv -(x_a-1)(x_b-1), \\
  \partial_i [x_i,x_a,x_i]\equiv (x_i-1)(x_a-1), \\
  \partial_i [x_a,x_i,x_i]\equiv -(x_i-1)(x_a-1), \\
  \partial_i [x_i,x_a,x_a]\equiv (x_a-1)^2, \\
  \partial_i [x_a,x_i,x_a]\equiv -(x_a-1)^2. \\
\end{array}
\eqno{(2)}
$$
Here different letters denote different indices and the symbol $\equiv$ means equality  modulo the ideal $R$.
For example,

\begin{eqnarray}
\partial_i [x_i,x_a,x_b] & = & \partial_i \left( [x_a,x_i]x_b^{-1} [x_i,x_a]x_b\right) \nonumber\\
& = &  \partial_i [x_a,x_i]+[x_a,x_i]x_b^{-1}\partial_i   [x_i,x_a] \nonumber\\
& = & \partial_i \left(x_a^{-1}x_i^{-1}x_ax_i\right)+   [x_a,x_i]x_b^{-1}\partial_i  \left(x_i^{-1}x_a^{-1}x_ix_a\right) \nonumber\\
& = & -x_a^{-1}x_i^{-1}+ x_a^{-1}x_i^{-1}x_a+    [x_a,x_i]x_b^{-1} \left(-x_i^{-1}+ x_i^{-1}x_a^{-1}\right) \nonumber\\
& = & x_a^{-1}x_i^{-1}\left( x_a-1 \right)+    \left([x_a,x_i]-1\right)x_b^{-1}x_i^{-1}x_a^{-1} \left( 1-x_a \right)+       x_b^{-1}x_i^{-1}x_a^{-1} \left( 1-x_a \right) \nonumber\\
& \equiv &  x_a^{-1}x_i^{-1}\left( x_a-1 \right)+         x_b^{-1}x_i^{-1}x_a^{-1} \left( 1-x_a \right)  \nonumber\\
& = & x_a^{-1}x_i^{-1}\left( x_a-1 \right)+x_b^{-1}(1-x_b)x_i^{-1}x_a^{-1} \left( 1-x_a \right)         x_i^{-1}x_a^{-1} \left( 1-x_a \right)  \nonumber\\
&  \equiv  & (x_a-1)(x_b-1)+(x_a^{-1}x_i^{-1}-x_i^{-1}x_a^{-1})(x_a-1)\equiv (x_a-1)(x_b-1).\nonumber
\end{eqnarray}

As a consequence of the above observations, we obtain the following example which establishes the existence of wild automorphisms of free nilpotent groups.
\begin{example}
For example, consider the case of $\N_{2, 3}$. Take the automorphism
$$\phi : \left\{
\begin{array}{ll}
x_1 \longmapsto \overline{[x_2, x_1]} x_1 [x_2, x_1] &  \\
x_2 \longmapsto x_2. &
\end{array} \right.
$$
We claim that this automorphism is wild. Note that $\overline{[x_2, x_1]} x_1[x_2, x_1]= x_1  [x_1, x_2, x_1]$.  But
$$\partial_1[x_1, x_2, x_1]=(x_1-1) (x_2-1) \neq 0, \ \partial_2 1=0,$$
this contradicts $(1)$ above. Therefore $\phi$ must be wild. Thus, the elements $\mu_{ij}$, where $1 \leq i \neq j \leq n$ do not provide a complete set of generators for $E \Pi A(\N_{2, 3})$. \end{example}

The following remark will be useful.

\begin{remark} \label{r2}
For each index  $i=1, \ldots, n$ and each pair of  words $u,v$ on the letters  $x_1,\ldots,x_{i-1}$, $ x_{i+1}$, $\ldots, x_n$ in the free group  $F_n$, the map

$$
\phi: \left\{
\begin{array}{l}
x_i \longmapsto ux_i v\\
x_j \longmapsto x_j~\textrm{for}~ j \neq i
\end{array} \right.
$$
is an automorphism of $F_n$.
\end{remark}

By \propref{cgen}, the group of central palindromic automorphisms of  $\N_{n,3}$ is generated by automorphisms  $\varphi_{ab,i}$ and $\varphi_{abc,i}$, where $1\leq b < a \leq n$, $1 \leq b \leq c \leq n$ and $1 \leq i \leq n$. Recall that, these act on the generators  $x_1, \ldots, x_n$ in the following manner
$$
\varphi_{ab,i} : \left\{
\begin{array}{l}
x_i \longmapsto x_i [x_a, x_b,x_i][x_a, x_b,x_b] [x_a,x_b,x_a] \\
x_j \longmapsto x_j  ~\mbox{for}~j \not= i
\end{array} \right.
$$
and
$$
\varphi_{abc,i} : \left\{
\begin{array}{l}
x_i \longmapsto x_i [x_a, x_b,x_c]^2  \\
x_j \longmapsto x_j  ~\mbox{for}~j \not= i.
\end{array} \right.
$$
Further, we will assume  that the  indices  $a,b,c,i$ and so on are arbitrary and not necessarily satisfy the above inequalities. Under this assumption, the corresponding automorphisms are central palindromic, but not necessary independent.

Next, we decide which of these automorphisms are tame.

\begin{lemma} \label{l5}
An automorphism  $\varphi_{ab,i}$ is tame if and only if the  indices  $a,b,i$ are all different.
\end{lemma}

\begin{proof}
The tame condition (1) for the automorphism  $\varphi_{ab,i}$ has the form
$$
\partial_i [x_a,x_b,x_i]+
    \partial_i [x_a,x_b,x_a]+
       \partial_i [x_a,x_b,x_b] \equiv 0.
$$
If $a=i \neq b$, then using  (2), we get
$$
\partial_i [x_a,x_b,x_i]+
    \partial_i [x_a,x_b,x_a]+
       \partial_i [x_a,x_b,x_b] \equiv 2 (x_i-1)(x_b-1)+(x_b-1)^2 \not\equiv 0.
$$
If $a\neq b=i$,  then using  (2), we get
$$
\partial_i [x_a,x_b,x_i]+
    \partial_i [x_a,x_b,x_a]+
       \partial_i [x_a,x_b,x_b] \equiv -2 (x_a-1)(x_i-1)-(x_a-1)^2 \not\equiv 0.
$$
If $a\neq i\neq b$,  then using  (2), we get
$$
\partial_i [x_a,x_b,x_i]+
    \partial_i [x_a,x_b,x_a]+
       \partial_i [x_a,x_b,x_b] \equiv  0.
$$
By Remark, in this case $\varphi_{ab,i}$ is tame since

\begin{eqnarray}
x_i^{\varphi_{ab,i}} & = & x_i [x_a, x_b,x_i][x_a, x_b,x_b] [x_a,x_b,x_a] \nonumber\\
& = &  x_i [x_a, x_b^{-1},x_i^{-1}][x_a, x_b,x_bx_a]  \nonumber\\
& = & [x_a, x_b^{-1},x_i^{-1}]x_i [x_a, x_b,x_bx_a] \nonumber\\
& = & [x_a, x_b^{-1}]^{-1}x_i [x_a, x_b^{-1}][x_a, x_b,x_bx_a]. \nonumber
\end{eqnarray}
This proves the lemma.
\end{proof}

\begin{lemma} \label{l6}
The automorphism  $\varphi_{abc,i}$  is tame if and only if  $a,b,c,i$  or $a,b,c=i$ are all different.
\end{lemma}

\begin{proof}
The tame condition  (1) for $\varphi_{abc,i}$ has the form
$$
\partial_i [x_a,x_b,x_c]^2\equiv 0.
$$
In other words,
$$
\partial_i [x_a,x_b,x_c] \equiv 0.
$$
By (2), the automorphism  $\varphi_{abc,i}$ is  tame if $a,b,c,i$  or $a,b,c=i$ are all different  indices. For different $a,b,c,i$, the automorphism  $\varphi_{abc,i}$ is evidently tame. For different $a,b,c=i$, we have
$$
 x_i^{\varphi_{abi,i}}= x_i [x_a, x_b,x_i]^2=
 x_i [x_a^2, x_b^{-1},x_i^{-1}] =
 [x_a^2, x_b^{-1},x_i^{-1}] x_i=
    [x_a^2, x_b^{-1}]^{-1}x_i [x_a^2, x_b^{-1}]
$$
which is a tame automorphism.
\end{proof}
\medskip

Let $\varphi$ be a central palindromic automorphism of  $\N_{n,3}$ for which condition (1) holds. By  \lemref{l5} and \lemref{l6}, it is enough to assume that
$\varphi$ is a product of automorphisms
$$
\varphi_{ai,i},\,\,\, \varphi_{ib,i},\,\,\, \varphi_{aii,i},\,\,\,   \varphi_{ibi,i},\,\,\, \varphi_{ibc,i}.
$$
Since
$$
 \varphi_{ai,i}^{-1}=\varphi_{ia,i}~\textrm{and}~
  \varphi_{aii,i}^{-1}= \varphi_{iai,i},
$$
it is enough to assume that $\varphi$ is a product of automorphisms
$$
\varphi_{ai,i},\,\,\, \varphi_{aii,i},\,\,\,  \varphi_{ibc,i}.
$$
Further, note that
$$
 x_i^{\varphi_{ai,i}}= x_i^{\varphi_{aii,i}} [x_a, x_i,x_a].
$$
Hence, if we introduce the automorphism  $\psi_{ai,i}$ given by
$$
\psi_{ai,i} : \left\{
\begin{array}{l}
x_i \longmapsto x_i [x_a, x_i,x_a]\\
x_j \longmapsto x_j  ~\mbox{for}~j \not= i,
\end{array} \right.
$$
then it is enough to consider the product of automorphisms
$$
\psi_{ai,i},\,\,\, \varphi_{aii,i},\,\,\,  \varphi_{ibc,i}.
$$
Since
$$
\psi_{ai,i}^2=\varphi_{iaa,i}^{-1},
$$
it follows that
$$
\varphi = \prod\limits_{i=1}^{n} \varphi_i,
$$
where
$$
\varphi_i =
            \prod\limits_{a\neq i} \left( \psi_{ai,i}^{A(a,i)} \varphi_{aii,i}^{B(a,i)}  \right)
            \prod\limits_{b\neq c\neq i}  \varphi_{ibc,i}^{D(b,c,i)},
$$
and $A(a,i)$, $B(a,i)$, $D(b,c,i) \in \mathbb{Z}$.

Since
$$
x_i^{\varphi}=x_i^{\varphi_i}=
         x_i
            \prod\limits_{a\neq i} \left( [x_a,x_i,x_a]^{A(a,i)} [x_a,x_i,x_i]^{B(a,i)}  \right)
            \prod\limits_{b\neq c\neq i}  [x_i,x_b,x_c]^{D(b,c,i)},
$$
condition  (1) has the form
$$
 \sum\limits_{i=1}^{n}
    \sum\limits_{a\neq i} \left( A(a,i)(x_a-1)^2+ B(a,i)(x_a-1)(x_i-1)  \right)
$$
$$
 \equiv \sum\limits_{i=1}^{n}
    \sum\limits_{b\neq c\neq i} D(b,c,i) (x_b-1)(x_c-1).
$$
This is further equivalent to the system

$$
 \sum\limits_{i=1}^{n} \sum\limits_{a\neq i} A(a,i)(x_a-1)^2\equiv 0
 \eqno{(3)}
$$
and
$$
 \sum\limits_{i=1}^{n}
    \sum\limits_{a\neq i} B(a,i)(x_a-1)(x_i-1)  \equiv
  \sum\limits_{i=1}^{n}
    \sum\limits_{b\neq c\neq i} D(b,c,i) (x_b-1)(x_c-1).
     \eqno{(4)}
$$

Recall that $S_n$ is a subgroup of $\Pi A(\N_{n,k})$ which acts on the generators  $x_1, \ldots, x_n$ in the following manner
$$
x_i^{\sigma}=x_{\sigma(i)},~\textrm{where}~1 \leq i \leq n~\textrm{and}~ \sigma\in S_n.
$$

The following lemma holds.

\begin{lemma} \label{l7}
The subgroup of  $\N_{n,3}$ whose elements satisfy the relation
$$
 \sum\limits_{i=1}^{n} \sum\limits_{a\neq i} A(a,i)(x_a-1)^2\equiv 0,
 \eqno{(3)}
$$
is generated by the automorphisms
$$
(\psi_{12,2}\psi_{13,3}^{-1})^{\sigma},~\textrm{where}~\sigma\in S_n.
$$
In particular, if $n=2$, then this subgroup is trivial.
\end{lemma}

\begin{proof}
First, we consider the case $n=2$. The relation (3) has the form
$$
 A(1,2)(x_1-1)^2+A(2,1)(x_2-1)^2\equiv 0.
$$
Therefore  $A(1,2)=A(2,1)=0$.

Let  $n\geq 3$. Put  $A(i,i)=0$ for $1 \leq i \leq  n$ and rewrite relation in equivalent form of the system of linear equations
$$
 \sum\limits_{i=1}^{n}  A(a,i)= 0,~\textrm{where}~1 \leq a \leq n.
$$
Corresponding automorphisms
$$
 \psi_{a}= \prod\limits_{i=1}^{n}  \psi_{ai,i}^{A(a,i)},~\textrm{where}~1 \leq a \leq n
$$
act on $x_1, \ldots, x_n$ by the formula
$$
x_i^{\psi_{a}}=x_i [x_a,x_i,x_a]^{A(a,i)}~\textrm{for}~ 1 \leq i\leq n.
$$
A permutation  $\sigma\in S_n$ acts on $\psi_{a}$ by the rule
$$
\psi_{a}^{\sigma}=\psi_{\sigma(a)},~\textrm{where}~ 1 \leq a \leq n.
$$
Thus it is enough to consider only the automorphism
$$
 \psi_{1}= \prod\limits_{i=1}^{n}  \psi_{1i,i}^{A(1,i)},~\textrm{where}~ \sum\limits_{i=1}^{n}  A(1,i)= 0.
$$
A vector
$$
 \big(A(1,2), \ldots , A(1,n)\big)\in \mathbb{Z}^{n-1}
$$
such that  $ \sum\limits_{i=1}^{n}  A(1,i)= 0$ is a linear combination of vectors
$$
 (1,-1,0,0, \ldots, 0,0),\,\,\,(0,1,-1,0, \ldots , 0,0), \ldots, (0,0,0,0, \ldots , 1,-1).
$$
Hence, in the generating set, it is enough to include automorphisms
$$
\psi_{1i,i}\psi_{1,i+1,i+1}^{-1},~\textrm{where}~ i=2, \ldots, n-1.
$$
The cycle   $\sigma=(23 \ldots n)$ acts on them as follows
$$
(\psi_{1i,i}\psi_{1,i+1,i+1}^{-1})^{\sigma}=\psi_{1,i+1,i+1}\psi_{1,i+2,i+2}^{-1},~\textrm{where}~ i=2, \ldots, n-2.
$$
Hence, modulo the action of $S_n$, we can take only one automorphism  $\psi_{12,2}\psi_{13,3}^{-1}$.
\end{proof}

\begin{lemma} \label{l8}
A subgroup of the automorphism group of $\N_{n,3}$ that satisfy the relation
$$
 \sum\limits_{i=1}^{n}
    \sum\limits_{a\neq i} B(a,i)(x_a-1)(x_i-1)  \equiv
  \sum\limits_{i=1}^{n}
    \sum\limits_{b\neq c\neq i} D(b,c,i) (x_b-1)(x_c-1)
     \eqno{(4)}
$$
for $n\geq3$ is generated by automorphisms
$$
 (\varphi_{231,1})^{\sigma},\,\,\,
  (\varphi_{123,1}\varphi_{322,2})^{\sigma},~\textrm{where}~\sigma\in S_n.
$$
For  $n=2$, this subgroup is generated by a single automorphism $\varphi$ with
$$
x_1^{\varphi}=x_1[x_2,x_1,x_1]^2~\textrm{and}~ x_2^{\varphi}=x_1[x_2,x_1,x_2]^2,
$$
which is inner.
\end{lemma}

\begin{proof}
First, we consider the case $n=2$. The relation (4) has the form
$$
 B(2,1)(x_2-1)(x_1-1)+ B(1,2)(x_1-1)(x_2-1) \equiv 0.
$$
Hence, the corresponding automorphism has the form
$$
\varphi=\varphi_{211,1}^{B(2,1)} \varphi_{122,2}^{B(1,2)},~\textrm{where}~B(2,1)+ B(1,2) = 0
$$
and is a power of the inner automorphism
$$
\varphi : \left\{
\begin{array}{l}
x_1 \longmapsto x_1[x_2,x_1,x_1]^2=x_1^{[x_1,x_2]^2}\\
x_2 \longmapsto x_1[x_2,x_1,x_2]^2=x_2^{[x_1,x_2]^2}.
\end{array} \right.
$$

Let  $n\geq 3$. The relation (4) is equivalent to the system of linear equations
$$
  B(a,i)+ B(i,a) =  \sum\limits_{k\neq a,i} \left( D(a,i,k)+D(i,a,k) \right),~\textrm{where}~  1 \leq a \neq i \leq n.
     \eqno{(5)}
$$
This system of equations has $(n^2-n)/2$ equations, one equation for each pair of different  indices  $a\neq i$. Let us fix a pair $a\neq i$. The automorphism which corresponds to relation  (5) has the form
$$
\varphi =
           \varphi_{aii,i}^{B(a,i)}  \varphi_{iaa,a}^{B(i,a)}
             \prod\limits_{k\neq a,i}  \varphi_{kai,k}^{D(a,i,k)}
               \prod\limits_{k\neq a,i}  \varphi_{kia,k}^{D(i,a,k)}
$$
and act on  $x_1, \ldots, x_n$ in the following manner
$$
\varphi : \left\{
\begin{array}{l}
x_k \longmapsto x_k [x_k,x_a,x_i]^{2D(a,i,k)}[x_k,x_i,x_a]^{2D(i,a,k)}~\textrm{for}~ k\neq a,i\\
x_a \longmapsto x_a [x_i,x_a,x_a]^{2B(i,a)}\\
x_i \longmapsto x_i [x_a,x_i,x_i]^{2B(a,i)}.
\end{array} \right.
$$
Using the Jacobi identity, we get
$$
[x_k,x_i,x_a][x_i,x_a,x_k][x_a,x_k,x_i]=1.
$$
We can rewrite the action of $\varphi$ in the form
$$
\varphi : \left\{
\begin{array}{l}
x_k \longmapsto x_k [x_k,x_a,x_i]^{2(D(a,i,k)+D(i,a,k))}[x_a,x_i,x_k]^{2D(i,a,k)}~\textrm{for}~ k\neq a,i\\
x_a \longmapsto x_a [x_i,x_a,x_a]^{2B(i,a)}\\
x_i \longmapsto x_i [x_a,x_i,x_i]^{2B(a,i)}.
\end{array} \right.
$$
The map
$$
\begin{array}{l}
x_k \longmapsto x_k [x_a,x_i,x_k]^{2D(i,a,k)}=x_k^{[x_i,x_a]^{2D(i,a,k)}}~\textrm{for}~ k\neq a,i\\
x_a \longmapsto x_a\\
x_i \longmapsto x_i
\end{array}
$$
is a tame automorphism (central and palindromic). It can be written as a product
$$
    \prod\limits_{k\neq a,i}  \varphi_{aik,k}^{D(i,a,k)}.
$$
It is not difficult to see that the automorphism  $\varphi_{aik,k}$ conjugates to the automorphism  $\varphi_{231,1}$ by some permutation in $S_n$.
Going modulo these automorphisms, we have
$$
\varphi : \left\{
\begin{array}{l}
x_k \longmapsto x_k [x_k,x_a,x_i]^{2D(a,i,k)}~\textrm{for}~ k\neq a,i\\
x_a \longmapsto x_a [x_i,x_a,x_a]^{2B(i,a)}\\
x_i \longmapsto x_i [x_a,x_i,x_i]^{2B(a,i)},
\end{array} \right.
$$
where $B(a,i)+ B(i,a) =  \sum\limits_{k\neq a,i} D(a,i,k)$.

Evidently,  the automorphism  $\varphi$ is a product of automorphisms of the following type
$$
\varphi : \left\{
\begin{array}{l}
x_k \longmapsto x_k [x_k,x_a,x_i]^{2D(a,i,k)}~\textrm{for}~ k\neq a,i\\
x_a \longmapsto x_a\\
x_i \longmapsto x_i [x_a,x_i,x_i]^{2B(a,i)}
\end{array} \right.
$$
with $B(a,i) =  \sum\limits_{k\neq a,i} D(a,i,k)$ and

$$
\varphi : \left\{
\begin{array}{l}
x_k \longmapsto x_k [x_k,x_a,x_i]^{2D(a,i,k)}~\textrm{for}~ k\neq a,i\\
x_a \longmapsto x_a [x_i,x_a,x_a]^{2B(i,a)}\\
x_i \longmapsto x_i,
\end{array} \right.
$$
where $B(i,a) =  \sum\limits_{k\neq a,i} D(a,i,k)$.

Note that automorphisms of this form are conjugate by some permutation from  $S_n$. For example, we can take  $\sigma=(12)$. Therefore it is enough to consider only one of them, for example, the second. Moreover, modulo the action of  $S_n$, we can assume that  $a=n-1$ and $i=n$.
Hence, we have
$$
\varphi : \left\{
\begin{array}{l}
x_k \longmapsto x_k [x_k,x_{n-1},x_n]^{2D(n-1,n,k)}~\textrm{for}~ k=1, \ldots, n-2\\
x_{n-1} \longmapsto x_{n-1} [x_n,x_{n-1},x_{n-1}]^{2B(n,n-1)}\\
x_n\longmapsto x_n,
\end{array} \right.
$$
where $B(n,n-1) =  \sum\limits_{k=1}^{n-2} D(n-1,n,k).$ This automorphism is a product of powers of automorphisms
$$
\varphi : \left\{
\begin{array}{l}
x_k \longmapsto x_k [x_k,x_{n-1},x_n]^{2}~\textrm{for}~ k\neq n, n-1\\
x_{n-1} \longmapsto x_{n-1} [x_n,x_{n-1},x_{n-1}]^{2}\\
x_n\longmapsto x_n.
\end{array} \right.
$$
Each of them conjugate by some permutation to the automorphism $\varphi_{123,1}\varphi_{322,2}$.
\end{proof}

For convenience, recall that
$$
\partial_1 w_1+ \cdots +\partial_n w_n \equiv 0 \left( \mathrm{mod} ~ R \right).
\eqno{(1)}
$$

In view of \lemref{l5} and \lemref{l6}, we can formulate the main result of this section.

\begin{theorem}\label{central-bgml-condition}
The subgroup of central palindromic automorphisms of $\N_{n,3}$ which satisfy (1) for $n\geq3$ is generated by the automorphisms
$$
  (\varphi_{23,1})^{\sigma},\,\,\,
   (\varphi_{234,1})^{\sigma},\,\,\,
    (\varphi_{231,1})^{\sigma},\,\,\,
     (\psi_{12,2}\psi_{13,3}^{-1})^{\sigma},\,\,\,
      (\varphi_{123,1}\varphi_{322,2})^{\sigma},
$$
where $\sigma\in S_n$ and the automorphisms $\varphi_{23,1}$, $\varphi_{234,1}$, $\varphi_{231,1}$ are tame. 

For  $n=2$, this subgroup is generated by the single  automorphism $\varphi$ with
$$
x_1^{\varphi}=x_1[x_2,x_1,x_1]^2~\textrm{and}~x_2^{\varphi}=x_1[x_2,x_1,x_2]^2,
$$
which is tame.
\end{theorem}

\section{Some Problems}
Next, we discuss some open problems regarding palindromic automorphisms of free nilpotent groups. We define a filtration of $E \Pi A (\N_{n, k})$ as follows
$$PI_l(\N_{n, k})=\bigg \{\phi \in E \Pi A(\N_{n, k}) \ | \ \phi: x_i \mapsto \overline{q_i} x_i q_i, where \ q_i \in \gamma_l \N_{n, k} ~for \ i=1, \ldots, n \bigg\}.$$
It follows from the proof of Theorem \ref{group} that $PI_l(\N_{n, k})$ is a group for each $l$. Note that $E \Pi A(\N_{n, k})=PI_1(\N_{n, k})$, $PI(\N_{n, k})=PI_2(\N_{n, k})$ and we have
$$E \Pi A(\N_{n, k})=PI_1(\N_{n, k}) \geq PI_2(\N_{n, k})\geq  \cdots \geq PI_k(\N_{n, k}) \geq 1.$$

Also, we can take the lower central series of $PI(\N_{n, k})$
$$PI(\N_{n, k})=\gamma_1 PI(\N_{n, k}) \geq \gamma_2 PI(\N_{n, k}) \geq \cdots $$
Note that $PI_2(\N_{n, k})=\gamma_1 PI(\N_{n, k})$. It would be interesting to see connections between the groups $\gamma_s PI(\N_{n, k})$ and $PI_l(\N_{n, k})$ for $s,l \geq 1$.

Obtaining generators and relations of the groups $PI_l(\N_{n, k})$ would provide better understanding of these groups. For $l=2$ and $n=3$, \propref{cgen} gives a generating set. For $l=3$ and $n=3$, the group $PI(\N_{n, 3})$ is generated by the automorphisms $\psi_{abc, i}$. We see from  \thmref{central-palin} that if $l$ is even, then $PI_l(\N_{n, l})$ is trivial. Also, it follows from the proof of \thmref{central-palin} that, if $l$ is odd, then $PI_l(\N_{n, l})$ is non-trivial and generated by the automorphisms of the form $x_i \mapsto x_i[y_1, \ldots, y_l]^2$, where $y_1, \ldots, y_l$ are arbitrary elements of $\N_{n, l}$.

In \cite{BGS}, generators of $E \Pi A(F_3)' \cap IA(F_3)$ were obtained. In view of \cite[Proposition 6.6]{BGS}, we conjecture the following.

\begin{conjecture}
The elements of the form $[\mu_{ik}, \mu_{ij}]^{\mu_{lm}}$ generate $PI(\N_{n, 3})$.
\end{conjecture}

We conclude with some more open problems.

\begin{problem}
Are the automorphisms $ \psi_{12,2}\psi_{13,3}^{-1}$ and $ \varphi_{123,1}\varphi_{322,2}$  tame?
\end{problem}

\begin{problem}
Describe the intersection of the group of normal automorphisms and the group of palindromic automorphisms of free nilpotent groups.
\end{problem}

\begin{problem}
Describe the intersection of the group of pointwise inner automorphisms (class preserving automorphisms) and  the group of palindromic automorphisms of free nilpotent groups.
\end{problem}
\medskip

\subsection*{Acknowledgements. }The authors gratefully acknowledge the support from the DST-RFBR  project DST/INT/RFBR/P-137 and RFBR-13-01-92697. Bardakov is partially supported by Laboratory of Quantum Topology of Chelyabinsk State University via RFBR grant 14.Z50.31.0020 and 14-01-00014. Gongopadhyay is partially supported by NBHM grant NBHM/R. P.7/2013/Fresh/992. Neshchadim is partially supported by RFBR grant 14-01-00014. Singh is also supported by DST INSPIRE Scheme IFA-11MA-01/2011 and DST Fast Track Scheme SR/FTP/MS-027/2010.


\begin{thebibliography}{HD}
\bibitem{Andreadakis} S. Andreadakis, {\it On the automorphisms of free groups and free nilpotent groups},  Proc. London Math. Soc. 15 (1965), 239--268.
\bibitem{BG2} V. G. Bardakov and K. Gongopadhyay, {\it On palindromic width of certain extensions and quotients of free nilpotent groups}, Internat. J. Algebra Comput. 24 (2014), 553--567.
\bibitem{BGS} V. G. Bardakov, K. Gongopadhyay and M. Singh, {\it  Palindromic automorphisms of free groups},  J. Algebra 438 (2015), 260--282.
\bibitem{Bir} J. S. Birman, \textit{Braids, Links, and Mapping Class Groups}, Annals of Math. Studies 82, Princeton University Press, 1974.
\bibitem{Bry} R. M. Bryant, C. K. Gupta, F. Levin and H. Y. Mochizuki, {\it Nontame automorphisms of free nilpotent groups}, Comm. Algebra 18 (1990), 3619--3631.
\bibitem{C} D. Collins, \textit{Palindromic automorphism of free groups}, Combinatorial and geometric group theory, London Math. Soc. Lecture Note Series 204, Cambridge Univ. Press, Cambridge 1995, 63--72.
\bibitem{DP} M. Day and A. Putman,  \textit{The complex of partial bases for $F_n$ and finite generation of the Torelli subgroup of $Aut(F_n)$}, Geom. Dedicata 164 (2013), 139--153.
\bibitem{End} G. Endimioni, {\it Pointwise inner automorphisms in a free nilpotent group}, Quart. J. Math. 53 (2002), 397--402.
\bibitem{Fox} R. H. Fox, {\it Free differential calculus I - Derivation in the free group ring}, Ann. Math. 57 (1953), 547--560.
\bibitem{Fullarton} N. Fullarton, {\it A generating set for the palindromic Torelli group}, Algebr. Geom. Topol. to appear.
\bibitem{gj} H. H. Glover and C. A. Jensen, \textit{Geometry for palindromic automorphism groups of free groups}, Comment. Math. Helv. 75 (2000), 644--667.
\bibitem{Gup} N. Gupta, {\it Free  group rings}, Contemporary Mathematics, 66. American Mathematical Society, Providence, RI, 1987. xii+129 pp.
\bibitem{jmm} C. Jensen, J.  McCammond, and J.  Meier, {\it The Euler characteristic of the Whitehead automorphism group of a free product}.
Trans. Amer. Math. Soc. 359 (2007), 2577--2595.
\bibitem{L} A. N. Lyul'ko, {\it Normal'nye avtomorfizmy svobodnyh nilpotentnyh grupp}, Voprosy teorii algebraicheskih sistem, Karaganda. (1981), 49--54.
\bibitem{Pas} I. B. S. Passi, {\it Group Rings and Their Augmentation Ideals}, Lecture Notes in Mathematics 715, Springer-Verlag, Berlin-Heidelberg-New York, 1979.
\bibitem{pr} A. Piggott, K. Ruane, \textit{Norml forms for automorphisms of universal coxeter groups and palindromic automorphisms of free groups},  Internat. J. Algebra Comput. 20 (2010), 1063--1086.
\bibitem {Magnus} W. Magnus, A. Karrass and D. Solitar, \textit{Combinatorial group theory}, Interscience Publishers, New York, 1996.
 \bibitem{ne1} A. I. Nekritsukhin,  {\it On some properties of palindromic automorphisms of a free group}. (Russian) Chebyshevski Sb. 15 (2014), no. 1 (49), 141--145.
 \bibitem{ne2}  A. I. Nekritsukhin, {\it  Palindromic automorphisms of a free group}. (Russian) Chebyshevski Sb. 9 (2008), no. 1(25), 148--152.
\end{thebibliography}
\end{document}